\documentclass[a4paper,10pt,twoside]{article}
\usepackage[left=3cm, right=3cm, top=3cm, bottom=3cm]{geometry}
\usepackage{type1cm}
\usepackage[utf8]{inputenc}
\usepackage[T1]{fontenc}
\usepackage{aecompl}
\usepackage{ae}
\usepackage{textcomp}
\usepackage[english]{babel}
\usepackage{pdfsync}
\usepackage{tikz}
\usetikzlibrary{arrows,automata}
\usetikzlibrary{positioning}
\usetikzlibrary{calc}
\usepackage{pgffor}
\usepackage{dsfont}
\usepackage{capt-of}
\usepackage{amssymb}
\usepackage{varwidth}
\usepackage{multirow}
\usepackage{rotating}
\usepackage{booktabs}

\usepackage{subfig}
\usepackage[leftcaption]{sidecap}
\usepackage{caption}
\usepackage{hyperref}
\usepackage{hypbmsec}
\usepackage{amsmath,amsthm,amsfonts}
\usepackage{mathtools}
\usepackage{bbm}
\usepackage{enumerate}

\theoremstyle{plain}
\newtheorem{thm}{Theorem}[section]
\newtheorem{prop}[thm]{Proposition}
\newtheorem{cor}[thm]{Corollary}
\newtheorem{lem}[thm]{Lemma}
\newtheorem{rem}[thm]{Remark}
\newtheorem{defn}[thm]{Definition}

\theoremstyle{definition}
\newtheorem{exam}[thm]{Example}

\newcommand{\E}{\mathbb{E}}
\renewcommand{\P}{\mathbb{P}}

\newcommand{\N}{\mathbb{N}}
\newcommand{\R}{\mathbb{R}}

\newcommand{\D}{\mathcal{D}}
\newcommand{\T}{\mathcal{T}}

\usepackage{cmll}

\newcommand{\down}{\shpos}

\newcommand{\G}{\mathsf{G}}
\renewcommand{\H}{\mathsf{H}}

\renewcommand{\d}{\mathrm{d}}

\renewcommand{\r}{\mathsf{r}}

{\left[
\begin{aligned}}%
{\end{aligned}
\right.}
\usepackage{fancyhdr}
\hypersetup{%
  pdftitle = {Transience and recurrence of rotor-router walks on directed covers of graphs},
  pdfauthor = {Wilfried Huss, Ecaterina Sava-Huss},
  pdfkeywords = {rotor-router walks, transience, recurrence, directed cover}
}

\begin{document}
\title{Transience and recurrence of rotor-router walks on directed covers of graphs}
\author{Wilfried Huss\footnote{Graz University of Technology, Austria.}, 
Ecaterina Sava-Huss\footnote{Graz University of Technology, Austria. Research supported by the FWF program W 1230-N13.}}
\maketitle

\begin{abstract}
The aim of this note is to extend the result of {\sc Angel and Holroyd} 
\cite{angel_holroyd_2011} concerning the transience and the recurrence
of transfinite rotor-router walks, for random initial configuration of rotors
on homogeneous trees. We address the same question on directed covers of finite graphs,
which are also called trees with finitely many cone types or periodic trees.
Furthermore, we provide an example of a directed cover such that the rotor-router walk
can be either recurrent or transient, depending only on the planar embedding of the periodic tree.
\end{abstract}

\textbf{Keywords:} graphs, directed covers, rotor-router walks,
multitype branching process, recurrence, transience.

\textbf{Mathematics Subject Classification:} 05C05; 05C25; 82C20.

\section{Introduction}

Suppose we are given a finite strongly connected and directed graph $\G$ with
adjacency matrix $D=(d_{ij})_{i,j\in\G}$. Using $\G$, one can construct
a labelled rooted tree $\T$ in the following way. Start with a root vertex which is labelled
with some $i\in\G$. Then define the tree recursively such that, if $x$ is a vertex in $\T$
with label $i\in\G$, then $x$ has $d_{ij}$ successors with label $j$. The tree $\T$ is called the 
\emph{directed cover} of $\G$. Random walks on directed covers of graphs
have been studied by \textsc{Takacs} \cite{Takacs97randomwalk}, 
\textsc{Nagnibeda and Woess} \cite{woess_nagnibeda}. 
On infinite graphs, their methods have been extended by \textsc{Gilch and Müller} \cite{gilch_mueller_covers}.

\emph{Rotor-router walks} have been first introduced into the physics literature under the 
name \emph{Eulerian walks} by \textsc{Priezzhev, D.Dhar et al} \cite{PhysRevLett.77.5079}
as a model of \emph{self organized criticality}, a concept established by \textsc{Bak,
Tang and Wiesenfeld} \cite{back_tang_wiesenfeld}. To define a rotor-router walk on a graph
consider on each vertex of the graph an arrow (the rotor) pointing to one of the neighbours
of the vertex. A particle performing a rotor-router walk first changes the 
rotor at its current position to point to the next neighbour, in a fixed order chosen at the beginning, 
and then moves to the neighbour the rotor is now pointing at. These walks have received increased
attention in the last years, and in many settings there is remarkable agreement between
the behaviour of rotor-router walks and the expected behaviour of random walks. For example, see
\textsc{Holroyd and Propp} \cite{holroyd_propp}, who showed that many quantities
associated to rotor-router walks such as normalized hitting frequencies, hitting
times and occupation frequencies, are concentrated around their expected values
for random walks. 

For a bibliographical picture in this context, see also
{\sc Cooper and Spencer} \cite{cooper_spencer_2006},
{\sc Doerr and Friedrich} \cite{doerr_friedrich_2007}, 
{\sc Angel and Holroyd} \cite{angel_holroyd_2011}, 
and also \textsc{Cooper, Doerr et al.} \cite{cooper_doerr_friedrich_spencer_2006}. On the other hand,
rotor-router walks and random walks can also have striking differences. 
For example, in questions concerning recurrence and transience of rotor-router walks on 
homogeneous trees, this has been proven by {\sc Landau and Levine} \cite{landau_levine_2009}.
For random initial configurations on homogeneous trees, see {\sc Angel and Holroyd}
\cite{angel_holroyd_2011}.
Furthermore, one can use rotor-router walks in order to solve questions regarding the 
behaviour of random walks: for example, in \cite{huss_sava_rotor} we have used a special rotor-router
process in order to determine the harmonic measure, that is, the exit distribution
of a random walk from a finite subset of a graph.

In this note, we extend the result of {\sc Angel and Holroyd} 
\cite[Theorem 6]{angel_holroyd_2011} for rotor-router walks
with random initial configuration of rotors on directed covers of graphs.
The proofs are a generalization of \cite{angel_holroyd_2011}
and are based on the extinction/survival of an appropriate multitype branching process (MBP).
Such a MBP encodes the subtree on which rotor-router particles
can reach infinity. We also give several examples where different phase transitions 
may appear. We give a graph $\G$ with two types of vertices and consider its directed cover $\T$ with all its possible
planar embeddings in the plane. For the same random initial configuration of rotors on these trees,
we show that the behaviour of the rotor-router walk depends dramatically on the planar embedding.
This corresponds to the fact that different rotor sequence gives rise to different behaviour of the
rotor-router walk.

\section{Preliminaries}\label{sec:preliminaries}

\paragraph{Graphs and Trees.}
Let $\G=(V,E)$ be a locally finite and connected directed multigraph, with vertex set $V$
and edge set $E$. For ease of presentation, we shall identify the graph $\G$ with its vertex 
set $V$, i.e., $i\in \G$ means $i\in V$. If $(i,j)$ is an edge of $\G$, we write $i\sim_{\G}j$,
and write $\d(i,j)$ for the {\em graph distance}.
Let $D=(d_{ij})_{i,j\in\G}$ be the \emph{adjacency matrix} of $\G$, where $d_{ij}$ is
the number of directed edges connecting $i$ to $j$.
We write $d_i$ for the sum of the entries in the $i$-th row of $D$, that is
$d_{i}=\sum_{j\in\G} d_{ij}$ is the {\em degree} of the vertex $i$. 

A {\em tree} $\T$ is a connected, cycle-free graph. A {\em rooted tree} is a tree with a distinguished vertex
$r$, called {\em the root}. For a vertex $x\in\T$, denote by $|x|$ the {\em height} of $x$, that is the graph distance from the root to $x$.
For $h\in\N$, define the {\em truncated tree} $\T^h = \{x\in\T: |x| \leq h\}$ to be the subgraph of $\T$ induced by the
vertices at height smaller or equal to $h$. 
For a vertex $x\in \T\setminus\{r\}$, denote by $x^{(0)}$ its {\em ancestor}, that is
the unique neighbour of $x$ closer to the root $r$.
It will be convenient to attach an additional vertex $r^{(0)}$ to the root $r$, which will
be considered in the following as a sink vertex, which we call the \emph{down sink} $s_\down=r^{(0)}$.
Additionally we fix a planar embedding of $\T$ and enumerate
the neighbours of a vertex $x\in \T$ in counter-clockwise order $\big(x^{(0)}, x^{(1)}, \ldots, x^{(d_x-1)}\big)$
beginning with the ancestor.
We will call a vertex $y$ a \emph{descendant} of $x$, if $x$ lies on the unique shortest path from $y$
to the root $r$. A descendant of $x$, which is also a neighbour of $x$, will be called a \emph{child}.
The \emph{principal branches} of $\T$ are the subtrees rooted at the children of the root $r$.


\paragraph{Directed Covers of Graphs.}
Suppose now that $\G$ is a finite, directed and strongly connected multigraph with  adjacency matrix $D=(d_{ij})$. 
Let $m$ be the cardinality of the vertices of $\G$, and label the vertices of $\G$ by $\{1,2,\ldots, m \}$.
The {\em directed cover} $\T$ of $\G$ is defined recursively as a rooted tree
$\T$ whose vertices are labelled by the vertex set $\{1,2,\ldots, m \}$ of $\G$.
The root $r$ of $\T$ is labelled with some $i\in\G$. Recursively, if $x$ is a vertex in $\T$ with
label $i\in \G$, then $x$ has $d_{ij}$ descendants with label $j$. We define
the {\em label function} $\tau:\T\to\G$ as the map that associates to each vertex in $\T$ its label 
in $\G$. The label $\tau(x)$ of a vertex $x$ will be also called the {\em type} of $x$. 
For a vertex $x\in\T$, we will not only need its type, but also the types of its children.
In order to keep track of the type of a vertex and the types of its children we
introduce the {\em generation function} $\chi = \left(\chi_i\right)_{i\in\G}$ with 
$\chi_i:\{1,\ldots,d_i\}\to\G$. For a vertex $x$ of type $i$,
$\chi_i(k)$ represents the type of the $k$-th child $x^{(k)}$ of $x$, i.e.,
\begin{equation*}
\text{if } \tau(x)=i \text{ then } \chi_i(k)=\tau(x^{(k)}), \text{ for } k=1,\ldots,d_i.
\end{equation*}
As the neighbours $\left(x^{(0)},\ldots, x^{(d_{\tau(x)})}\right)$ of any vertex $x$ are drawn in clockwise order,
the generation function $\chi$ also fixes the planar embedding of the tree and thus defines $\T$ uniquely
as a planted plane tree. The tree $\mathcal{T}$ constructed
in this way is called the \emph{directed cover} of $\G$. Such trees are also known as
\emph{periodic trees}, see {\sc Lyons} \cite{LP:book}, or \emph{trees with finitely many cone types} in 
{\sc Nagnibeda and Woess} \cite{woess_nagnibeda}. 
The graph $\G$ is called the \emph{base graph} or the \emph{generating graph} for the tree $\T$.
We write $\T_i$ for a tree with root $r$ of type $i$, that is $\tau(r)=i$.

\begin{figure}
\centering
\begin{tikzpicture}[scale=0.65]
\begin{scope}[xshift=-3cm, yshift=1cm]
\node (chi) at (0,0) {
\begin{tabular}{lc|cc}
\multicolumn{2}{c}{} & \multicolumn{2}{l}{\scalebox{0.7}{$k\rightarrow$}} \\
\multicolumn{2}{c|}{$\chi_i(k)$} & 1 & 2 \\
\hline
\multirow{2}{*}{\scalebox{0.7}{\rotatebox{-90}{\rotatebox{90}{$i$} $\rightarrow$}}}
  & 1 & 2 & \\
  & 2 & 2 & 1
\end{tabular}
};
\end{scope}
\begin{scope}[xscale=0.8]
\coordinate (n0) at (4.65625,0);
\coordinate (n1) at (4.65625,1);
\coordinate (n3) at (1.625,2);
\coordinate (n22) at (1.625,3);
\coordinate (n24) at (0.5,4);
\coordinate (n30) at (0.5,5);
\coordinate (n32) at (0,6);
\coordinate (n31) at (1,6);
\coordinate (n23) at (2.75,4);
\coordinate (n26) at (2.0,5);
\coordinate (n29) at (2,6);
\coordinate (n25) at (3.5,5);
\coordinate (n28) at (3,6);
\coordinate (n27) at (4,6);
\coordinate (n2) at (7.6875,2);
\coordinate (n5) at (5.75,3);
\coordinate (n16) at (5.75,4);
\coordinate (n18) at (5.0,5);
\coordinate (n21) at (5,6);
\coordinate (n17) at (6.5,5);
\coordinate (n20) at (6,6);
\coordinate (n19) at (7,6);
\coordinate (n4) at (9.625,3);
\coordinate (n7) at (8.5,4);
\coordinate (n13) at (8.5,5);
\coordinate (n15) at (8,6);
\coordinate (n14) at (9,6);
\coordinate (n6) at (10.75,4);
\coordinate (n9) at (10.0,5);
\coordinate (n12) at (10,6);
\coordinate (n8) at (11.5,5);
\coordinate (n11) at (11,6);
\coordinate (n10) at (12,6);
\end{scope}

\draw (n1) -- (n0);
\draw (n3) -- (n1);
\draw (n22) -- (n3);
\draw (n24) -- (n22);
\draw (n30) -- (n24);
\draw (n32) -- (n30);
\draw (n31) -- (n30);
\draw (n23) -- (n22);
\draw (n26) -- (n23);
\draw (n29) -- (n26);
\draw (n25) -- (n23);
\draw (n28) -- (n25);
\draw (n27) -- (n25);
\draw (n2) -- (n1);
\draw (n5) -- (n2);
\draw (n16) -- (n5);
\draw (n18) -- (n16);
\draw (n21) -- (n18);
\draw (n17) -- (n16);
\draw (n20) -- (n17);
\draw (n19) -- (n17);
\draw (n4) -- (n2);
\draw (n7) -- (n4);
\draw (n13) -- (n7);
\draw (n15) -- (n13);
\draw (n14) -- (n13);
\draw (n6) -- (n4);
\draw (n9) -- (n6);
\draw (n12) -- (n9);
\draw (n8) -- (n6);
\draw (n11) -- (n8);
\draw (n10) -- (n8);

\fill[black] (n0) circle (2pt);
\fill[red] (n1) circle (2pt);
\fill[blue] (n3) circle (2pt);
\fill[red] (n22) circle (2pt);
\fill[blue] (n24) circle (2pt);
\fill[red] (n30) circle (2pt);
\fill[blue] (n32) circle (2pt);
\fill[red] (n31) circle (2pt);
\fill[red] (n23) circle (2pt);
\fill[blue] (n26) circle (2pt);
\fill[red] (n29) circle (2pt);
\fill[red] (n25) circle (2pt);
\fill[blue] (n28) circle (2pt);
\fill[red] (n27) circle (2pt);
\fill[red] (n2) circle (2pt);
\fill[blue] (n5) circle (2pt);
\fill[red] (n16) circle (2pt);
\fill[blue] (n18) circle (2pt);
\fill[red] (n21) circle (2pt);
\fill[red] (n17) circle (2pt);
\fill[blue] (n20) circle (2pt);
\fill[red] (n19) circle (2pt);
\fill[red] (n4) circle (2pt);
\fill[blue] (n7) circle (2pt);
\fill[red] (n13) circle (2pt);
\fill[blue] (n15) circle (2pt);
\fill[red] (n14) circle (2pt);
\fill[red] (n6) circle (2pt);
\fill[blue] (n9) circle (2pt);
\fill[red] (n12) circle (2pt);
\fill[red] (n8) circle (2pt);
\fill[blue] (n11) circle (2pt);
\fill[red] (n10) circle (2pt);

\node (sdown) [right=0.125 of n0]  {$s_\down$};
\node (r) [right=0.125 of n1, yshift=-2]  {$r$};
\end{tikzpicture}
\caption{\label{fig:fib_tree} A finite piece of the Fibonacci tree $\T_2$ with root of type $2$, for the planar embedding given by $\chi$.}
\end{figure}
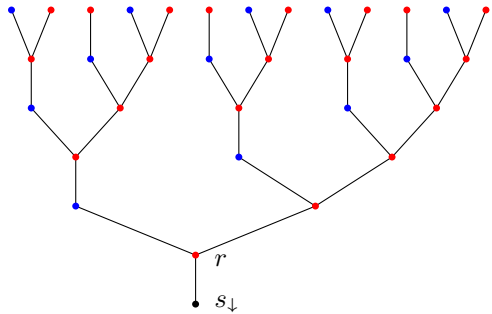

\begin{exam}The \emph{Fibonacci tree} is the directed cover of the
graph $\G$ on two vertices with adjacency matrix 
$
\begin{pmatrix}
0 & 1 \\
1 & 1 
\end{pmatrix}.
$
The \emph{ $(\alpha,\beta)$ bi-regular tree} with parameters $\alpha,\beta\in\N$ is the directed cover of the
graph $\G$ on two vertices with adjacency matrix 
$
\begin{pmatrix}
0 & \alpha \\
\beta & 0 
\end{pmatrix}
$.
\end{exam}


\paragraph{Rotor-Router Walks.}
\label{subsec:rr_walks}

On a locally finite and connected graph $\H$, a \emph{rotor-router walk} is defined as follows. 
For each vertex $x\in\H$ fix a cyclic ordering $c(x)$ of its neighbours: $c(x)=\big(x^{(0)},x^{(1)},\ldots,x^{(\deg(x)-1)}\big)$,
where  $x\sim_{\H} x^{(i)}$ for all $i=0,1,\ldots ,\deg(x)-1$ and $\deg(x)$ is the degree
of $x$. The ordering $c(x)$ is called the
\emph{rotor sequence} of $x$. A \emph{rotor configuration} is a function
$\rho: \H\to \H$, with $\rho(x) \sim_{\H} x$, for all $x\in \H$.
By abuse of notation, we write $\rho(x)=i$ if the rotor at $x$ points
to the neighbour $x^{(i)}$. A rotor-router walk is defined by the following rule.
Let $x$ be the current position of the particle, and $\rho(x) = i$ the state of
the rotor at $x$. In one step of the walk the following happens. First the position of the
rotor at $x$ is incremented to point to the next neighbour $x^{(i+1)}$ in the ordering $c(x)$, that is,
$\rho(x)$ is set to $i+1$ (with addition performed modulo $\deg(x)$). Then the
particle moves to position $x^{(i+1)}$. The rotor-router walk is obtained by repeatedly
applying this rule.

\paragraph{Recurrence and Transience} 
Suppose now that the graph $\H$ is 
infinite and connected, and let $o$ be a fixed vertex in $\H$, 
the root. Start with a particle at $o$ and let it perform
a rotor-router walk stopped at the first return to $o$. Either the particle
returns to $o$ after a finite number of steps ({\em recurrence}), or it escapes
to infinity without returning to $o$, and visiting each vertex only finitely 
many times ({\em transience}). In both cases, the positions of the rotors after
the walk is complete are well defined. Before starting a new particle at the root, we
do not reset the configuration of rotors. We then start a new particle at the root $o$,
and repeat the above procedure, and so on. This type of rotor-router walk is called
\emph{transfinite rotor-router walk}, see \textsc{Holroyd and Propp} \cite{holroyd_propp}.
Let $E_n=E_n(\H)$ be number of particles that escape
to infinity after $n$ rotor-router walks are run from $o$ in this way. 
The following result, due to {\sc Schramm} states that a rotor-router walk
is no more transient than a random walk. See \cite[Theorem 10]{holroyd_propp} for a proof.
\begin{thm}\label{thm:schramm_bound}[Density bound--Schramm] For any locally finite graph, any starting vertex $o$,
any cyclic order of neighbours and any initial rotor positions,
\begin{equation*}\label{eq:density_bound_schramm}
\limsup _{n\to\infty}\dfrac{E_n}{n}\leq \mathcal{E},
\end{equation*}
where $\mathcal{E}$ denotes the probability that a \emph{simple random walk}
on $\H$ started at the origin $o$ never returns to $o$.
\end{thm}

\section{Recurrence and Transience of Rotor-Router Walks}\label{sec:rec_trans_rrwalks}

We study now the behaviour of transfinite rotor-router walks on directed covers of graphs for
random initial rotor configurations. In particular, we generalize a theorem of
\textsc{Holroyd and Angel} \cite[Theorem 6]{angel_holroyd_2011} which proves a transition
between recurrent and transient phases for transfinite rotor-router walks on homogeneous 
trees $\mathbb{T}_b$ of degree $b$, with random initial rotor configuration $(\rho(v))_{v\in\mathbb{T}_b}$.
The random variables $\rho(v)$ are i.i.d on $\{0,1,\ldots,b\}$. 

For the tree $\T$ with root $r$, the quantity $E_n(\T,\rho)$ represents the number of particles that
escape to infinity after $n$ rotor-router walks are run from $r$ with initial rotor
configuration $\rho$, and $\mathcal{E}(\T)$ represents the probability that a simple random walk on
$\T$ started at $r$, never returns to $r$. The theorem then states the following.
\begin{thm}[Angel and Holroyd]
\label{thm:b_ary}
For a random i.i.d. initial rotor configuration $\rho$ on the homogeneous tree $\mathbb{T}_b$, writing
$v$ for an arbitrary vertex, we have almost surely
\begin{enumerate}[(i)]
 \item $\displaystyle\lim_{n\to\infty} \dfrac{E_n(\mathbb{T}_b,\rho)}{n} = \mathcal{E}(\mathbb{T}_b)$, if $\E\big[\rho(v)\big] < b-1$;
       \hfill(Transience)
 \item $E_n(\mathbb{T}_b,\rho) = 0$ for all $n\geq 0$, if $\E\big[\rho(v)\big] \geq b-1$.
       \hfill(Recurrence)
\end{enumerate}
\end{thm}
The discontinuous phase transition above is related to a branching process. 
The main idea of the proof is to model the connected subtree of vertices,
along which particles may move to infinity, as a Galton-Watson tree. In the case of directed covers,
since we have vertices of different types, we need to model a multitype branching process (MBP).

\paragraph{Multitype Branching Processes.}
A multitype branching process (MBP) is a generalization of a Galton-Watson process,
where  one allows a finite number of distinguishable types of particles with different
probabilistic behaviour. The particle types will coincide with the different types of vertices
in the direct covers under consideration, and will be denoted by $\{1,\ldots,m\}$.

A \emph{multitype branching process} is a Markov process $(\boldsymbol{Z}_n)_{n\in\N_0}$
such that the states $\boldsymbol{Z}_n = (Z^1_n,\ldots,Z^m_n)$ are $m$-dimensional vectors with non-negative components. The initial state
$\boldsymbol{Z}_0$ is nonrandom. The $i$-th entry $Z_n^i$ of $\boldsymbol{Z}_n$ represents the number of particles of type $i$ in the 
$n$-th generation. The transition law of the process is as follows. If $\boldsymbol{Z}_0=\boldsymbol{e}_i$,
where $\boldsymbol{e}_i$ is the $m$-dimensional vector whose $i-th$ component is $1$ and all the
others are $0$, then $\boldsymbol{Z}_1$ will have the generating function
$\boldsymbol{f}(\boldsymbol{z}) = \big(f^1(\boldsymbol{z}),\ldots,f^m(\boldsymbol{z})\big)$ with
\begin{equation}
\label{eq:MBP_generating_function}
 f^i(\boldsymbol{z}) = f^i(z_1,\ldots,z_m)
 = \smashoperator{\sum_{s_1,\ldots,s_m \geq 0}} p^i(s_1,\ldots,s_m)z_1^{s_1}\cdots z_m^{s_m},
\end{equation}
and $ 0\leq z_1,\ldots,z_m \leq 1$, where $p^i(s_1,\ldots,s_m)$ is the probability that a 
particle of type $i$ has $s_j$ children
of type $j$, for $j=1,\ldots,m$. For $\boldsymbol{i}=(i_1,\ldots,i_m)$
and $\boldsymbol{j}=(j_1,\ldots,j_m)$, the one-step transition probabilities are given by
\begin{equation*}
\boldsymbol{p}(\boldsymbol{i},\boldsymbol{j})
   =\P[\boldsymbol{Z}_{n+1}=\boldsymbol{j}|\boldsymbol{Z}_n=\boldsymbol{i}] 
   = \text{coefficient of } \boldsymbol{z}^{\boldsymbol{j}}
      \text{ in } \big(\boldsymbol{f}(\boldsymbol{z})\big)^{\boldsymbol{i}}
    := \prod_{k\in\G} f^k(\boldsymbol{z})^{i_k}. 
\end{equation*}
For vectors $\boldsymbol{z},\boldsymbol{s}$, we write $\boldsymbol{z}^{\boldsymbol{s}}=(z_1^{s_1},\ldots,z_m^{s_m})$.
Let $M=(m_{ij})$ be the matrix of the first moments:
\begin{equation}
\label{eq:MBP_first_moment}
 m_{ij}=\E[Z_1^j|\boldsymbol{Z}_0=\boldsymbol{e}_i] 
       = \left.\frac{\partial f^i(z_1,\ldots,z_m)}{\partial z_j}\right|_{\boldsymbol{z} = \boldsymbol{1}}
\end{equation}
represents the expected number of offsprings of type $j$ of a particle of type $i$ in one generation.
If there exists an $n$ such that $m_{ij}^{(n)}>0$ for all $i,j$, then $M$
is called {\em strictly positive} and the process $\boldsymbol{Z}_n$ is called
{\em positive regular}. If each particle has exactly one child, then $\boldsymbol{Z}_n$
is called {\em singular}. The following is well known; see {\sc Harris} \cite{harris}.
\begin{thm}\label{thm:surv/extinc}
 Assume $\boldsymbol{Z}_n$ is positive regular and nonsingular, and let $\r(M)$ be the spectral radius
 of $M$. If $\r(M)\leq 1$, then the process $\boldsymbol{Z}_n$ dies with probability one.
 If $\r(M)>1$, then $\boldsymbol{Z}_n$ survives with positive probability.
\end{thm}

\subsection{Nondeterministic Rotor Configurations on Directed Covers}

Recall the setting we are working on: $\G$ is a finite strongly connected graph with vertices labelled by $\{1,\ldots,m\}$;
for $i\in\G$, $\T_i$ is the directed cover of $\G$ with root  $r$ of type $i$. 
For a vertex $x\in\T_i$ with $\tau(x)=j\in\G$, we have denoted by $x^{(0)}$ its ancestor, and by
$x^{(1)},\ldots,x^{(d_j)}$  its $d_j$ children. We choose the
cyclic ordering $c(x)$ of the neighbours of $x$ to be $(x^{(0)},x^{(1)},\ldots,x^{(d_j)})$,
and we allow the initial rotor to point at an arbitrary neighbour in this order. We will embed
the tree in the plane in such a way that the rotors turn in counter-clockwise order, when following
this rotor sequence. Recall also that for some rotor configuration $\rho$ on $\T_i$, we write 
$\rho(x)=k$ if the rotor at $x$ points to the neighbour $x^{(k)}$.

Let now $\D=(\D_1,\ldots,\D_m)$ be a vector of probability distributions: for each $i\in\G$,
$\D_i$ is a probability distribution with values in $\{0,\ldots,d_i\}$. Consider 
a \emph{random initial configuration $\rho$} of rotors on $\T_i$, such that $\big(\rho(x)\big)_{x\in\T_i}$
are independent random variables, and $\rho(x)$ has distribution $\D_j$ if the vertex $x$ is of type $j$.
Shortly
\begin{equation}\label{eq:random_cfg}
 \rho(x) \stackrel{d}{\sim} \D_j \quad \Longleftrightarrow \quad \tau(x)=j.
\end{equation}
If \eqref{eq:random_cfg} is satisfied,
we shall say that the rotor configuration $\rho$ is $\D=(\D_1,\ldots,\D_m)$-distributed,
and we write $\rho\stackrel{d}{\sim}\D$.
Performing transfinite rotor-router walks on $\T_i$ with $\D$-distributed initial rotor configuration
$\rho$, we observe a phase transition between the recurrent and transient regimes, similar to
the case of homogeneous trees in Theorem \ref{thm:b_ary}.
For defining the critical point of this phase transition, we need to introduce some additional definitions.

Consider a general tree $\T$ with rotor configuration $\rho$. For a vertex $x\in\T$ define the set of \emph{good
children} as $\big\{x^{(k)}:\: \rho(x) < k \leq d_{\tau(x)}\big\}$. This means that a rotor-router particle starting
at a vertex will first visit all its good children before visiting its ancestor. An infinite sequence of vertices
$\big(x_n\big)_{n\in\N}$ with each vertex being a child of the previous one, is called a \emph{live path} if for every
$n\geq 0$ the vertex $x_{n+1}$ is a good child of $x_n$. An {\em end} of $\T$
is an infinite sequence of vertices $x_1,x_2,\ldots$ each being the ancestor of the next.
An end is called {\em live} if the subsequence $(x_i)_{i\geq j}$ starting at one of its vertices is a live path. 

Denote by $E_{\infty}(\T,\rho)=\lim_{n\to\infty}E_n(\T,\rho)$ the total number of particle escaping
to infinity, when one launches an infinite number of particles. Recall now an useful result for general trees $\T$, 
whose proof can be found in \cite[Proposition 8]{angel_holroyd_2011}.
\begin{prop}
\label{prop:live_ends}
The total number of escapes $E_{\infty}(\T,\rho)$ equals the number of live ends in the initial
rotor configuration $\rho$. 
\end{prop}

\begin{defn}
\label{defn:good_children}
For $i\in\G$ and $k\in\{0,\ldots,d_i\}$ denote by $\mathfrak{C}_i^j(k)$ the number of good children with type $j$ of
a vertex $x$ with type $i$, if the rotor $\rho(x)$ at $x$ is in position $k$, i.e.,
\begin{equation*}
\mathfrak{C}_i^j(k) = \#\big\{l\in\{k+1,\ldots,d_i\}:\: \chi_i(l) = j\big\}.
\end{equation*}
\end{defn}
We have that
$
 \sum_{j\in\G}\mathfrak{C}_i^j(k)=d_i-k.
$
Using this definition we can now define a MBP which models connected subtrees consisting of only good children.
In this MBP, $p^i(s_1,\ldots,s_m)$ represents the probability that a vertex of type
$i$ has $s_j$ good children of type $j$, with $j=1,\ldots,m$.
Define the generating function of the MBP as in \eqref{eq:MBP_generating_function} and the 
probabilities $p^i$  by
\begin{equation}\label{eq:mbp}
p^i(s_1,\ldots,s_m) = \begin{cases}
                      \D_i(k)\; &\text{if for all $j=1,\ldots,m:\: s_j=\mathfrak{C}_i^j(k)$, and $k\in\{0,\ldots,d_i\}$}, \\
                      0         &\text{otherwise},
                      \end{cases}
\end{equation}
with  $\D_i(k)=\P[\rho(x)=k]$,  for $k\in\{0,\ldots,d_i\}$ and $i\in\G$.
In the following we always make the additional assumption that that this MBP is positive
regular and nonsingular, such that Theorem \eqref{thm:surv/extinc} can be applied. In
particular when the rotors
point to every neighbour with positive probability these two conditions are always satisfied.
Let $M(\D)$ be the first moment matrix --- as defined in \eqref{eq:MBP_first_moment} --- of the MBP
with offspring probabilities given in \eqref{eq:mbp}. We are now ready to state our theorem,
as an extension of \cite[Theorem 6]{angel_holroyd_2011}.

\begin{thm}\label{thm:rec/trans_rr_walks}
Let $\rho$ be an initial random rotor configuration with distribution $\D=(\D_1,\ldots,\D_m)$ on 
the directed cover $\T_i$ with root $r$ of type $i$, of a finite graph $\G$ with $m$ vertices. Let $n$ particles
perform transfinite rotor-router walks on $\T_i$. Then we have almost surely:
\begin{enumerate}
 \item Recurrence: $E_n(\T_i,\rho)=0$, for all $n\geq 0$ and $i\in\G$ if $\r\big(M(\D)\big) \leq 1$;
 \item Transience: $\lim_{n\to\infty}\dfrac{E_n(\T_i,\rho)}{n}=\mathcal{E}_i$ for all $i\in\G$ if $\r\big(M(\D)\big) > 1$.
\end{enumerate}
\end{thm}
The quantity $\mathcal{E}_i$ represents the probability that a simple random walk starting at the root 
$r$ of $\T_i$ never returns to $r$, and $\r\big(M(\D)\big)$ is the spectral radius of $M(\D)$.

\begin{proof}[Proof of Theorem \ref{thm:rec/trans_rr_walks}(a)]
For any fixed $x\in\T_i$ with $\tau(x)=j\in\G$
the set of descendants of $x$ that can be reached via a path of good
vertices forms a multitype branching process with offspring distributions $p^j(s_1,\ldots,s_m)$
defined as in \eqref{eq:mbp}. The survival/extinction of this MBP is controlled by
the matrix of the first moments $M(\D)=(m_{ij})_{i,j\in\G}$.

Since $\r(M(\D))\leq 1$, the extinction probability is $1$ by Proposition \ref{thm:surv/extinc}
and the MBP dies almost surely, hence there are no live paths. Therefore by Proposition \ref{prop:live_ends}
there are no escapes almost surely and $E_n(\T_i,\rho)=0$. This gives the recurrence of the rotor-router
walk with random initial configuration $\rho$ which is $\D=(\D_1,\ldots,\D_m)$-distributed.
\end{proof}

\paragraph{Transience.}
The transience part in Theorem \ref{thm:rec/trans_rr_walks} requires some additional work.
We can assume that the direct cover $\T_i$ is not isomorphic to single infinite path, since
by Theorem \ref{thm:b_ary} we have recurrence in that case for any initial rotor
distribution.

\paragraph{The frontier rotor-router process $F_{\rho}(n)$.}
For a fixed rotor configuration $\rho$ consider the following process which generates a sequence
$F_\rho(n)$ of subsets of vertices of the tree. $F_{\rho}(n)$ is constructed by a rotor-router process consisting
of $n$ rotor-router walks starting at the root $r$, such that each vertex of $F_{\rho}(n)$ contains
exactly one particle.  In the first step put a particle at the root $r$ and set $F_{\rho}(1) = \{r\}$.
Inductively given $F_{\rho}(n)$ and the rotor configuration that was created by the previous step, we
construct the next set $F_{\rho}(n+1)$ using the following rotor-router procedure. Perform rotor-router
walk with a particle starting at the root $r$, until one of the following stopping conditions
occurs:
\begin{enumerate}[(a)]
\item\label{frontier_a} The particle reaches the down sink $s_\down$.
In this case we set $F_{\rho}(n+1) = F_{\rho}(n)$.
\item\label{frontier_b} The particle first reaches a vertex $x$, which has never been visited before.
In this case we set $F_{\rho}(n+1) = F_{\rho}(n) \cup \{x\}$.
\item\label{frontier_c} The particle reaches an element $y\in F_{\rho}(n)$. We delete $y$ from $F_{\rho}(n)$, i.e.
set $F'(n) = F_{\rho}(n)\setminus\{y\}$. At this time there are two particles at $y$, both of which are
restarted until stopping condition \eqref{frontier_a}, \eqref{frontier_b} or \eqref{frontier_c} for
the set $F'(n)$ applies to them. Note that since we are on a tree at least one particle will stop at a child of $y$ after one step, due
to halting condition \eqref{frontier_b}.
\end{enumerate}

We will call the set $F_{\rho}(n)$ the \emph{frontier} of $n$ particles. In the following we give several
properties of the frontier $F_{\rho}(n)$.

\begin{lem} The process generating $F_{\rho}(n)$ is always terminating in a finite number of steps and the set of vertices visited by
the particles during this process is finite.
\end{lem}
\begin{proof}
We prove this by induction. For $n=1$ the statement is obviously true.

Let $V(n)$ be the set of vertices visited while computing $F_{\rho}(n)$.
Assuming $V(n)$ is finite we know that after a finite number of steps the rotor-router walk first exits 
the set $V(n)$ at some vertex $x$. If $x = s_\down$ we are in case \eqref{frontier_a} and $V(n+1) = V(n)$. If
$x\not\in F_{\rho}(n)$ we are in case \eqref{frontier_b} and we have $V(n+1) = V(n) \cup \{x\}$.

Finally, if $x \in F_{\rho}(n)$ two particles are restarted at $x$. If both particles visit children of $x$ in
the first step, the process stops and these two vertices are added to $V(n)$. In the case that one of the
particles visits the ancestor of $x$ in the first step, this particle continues its walk until of the three
stopping conditions occurs. Since the first edge that is traversed in the same direction twice is an edge
emanating from the starting vertex (see \textsc{Angel and Holroyd} \cite[Lemma 8]{angel_holroyd_2012}) the
particle will enter the sink before it will return to $x$. Hence each vertex of $F_{\rho}(n)$ can be visited at
most once during the formation of $F_{\rho}(n+1)$. This means in particular that $V(n)$ is expanded by only a
finite number of vertices.
\end{proof}

\begin{rem}
Note that whenever a previously unexplored vertex is reached, it is immediately added to the set $F_{\rho}(n)$.
Hence $\max\{|x|:x\in V(n)\}= \max\{|x|:x\in F_{\rho}(n)\}$.
\end{rem}

\begin{defn}
For each vertex $x\in \mathcal{T}$ denote by
$\mathcal{C}(x) = \big\{y \in \mathcal{T}:\: y \text{ is a descendant of } x\big\} \cup \{x\}$
the \emph{cone} of $x$.
\end{defn}

\begin{lem}
$\mathcal{C}(x) \not= \mathcal{C}(y)$ for all $x,y\in F_{\rho}(n)$ with $x\not=y$. Let $\rho'$ be the rotor-router configuration at the end of the process generating $F_\rho(n)$. Then for all $x\in F_\rho(n)$
the rotor configuration in the cone of $x$ is unchanged, that is, 
$\rho_{|\mathcal{C}(x)} \equiv \rho'_{|\mathcal{C}(x)}$.
\end{lem}
\begin{proof}
This  follows immediately from the procedure generating the frontier.
\end{proof}
Let
\begin{equation}\label{eq:max_height}
M(n) = \max_{\rho} \max\big\{|x|: x\in F_\rho(n)\big\}. 
\end{equation}
be the maximal height of the frontier $F_{\rho}(n)$. 
We will need an upper bound for $M(n)$. Since whenever stopping condition \eqref{frontier_c} occurs,  the frontier moves
one level upwards at the vertex of $F_{\rho}(n)$ that was hit. We have the trivial upper bound of $M(n) \leq n$.
This bound is tight for general trees as shows the example of a single infinite path, where the frontier
$F_{\rho}(n)$ for $n\geq 2$ consists of a single vertex at distance $n-2$ from the root vertex, with the remaining
$n-1$ particles in the down sink $s_{\down}$.

In the case of directed covers with irreducible cone types which are not isomorphic to a single path, $M(n)$
seems to grow logarithmically in $n$ (see Figure \ref{fig:frontier}). For our purposes, a weak upper bound of the form $M(n) \leq c n$ for a constant $c < 1$ is sufficient.

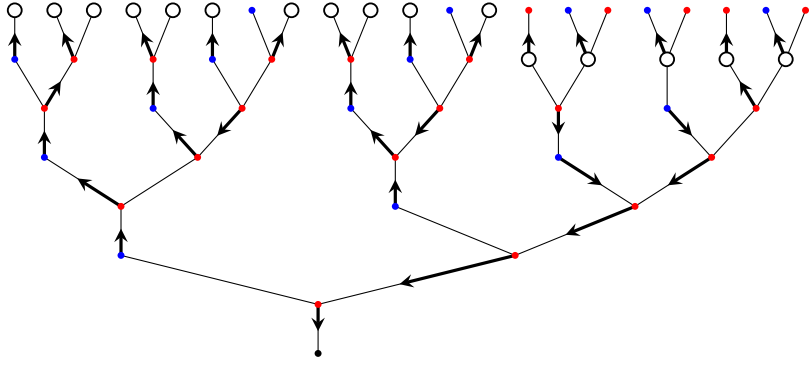
\begin{figure}
\centering
\begin{tikzpicture}[scale=0.65]
\begin{scope}[xscale=0.8]
\coordinate (n0) at (7.671875,0);
\coordinate (n1) at (7.671875,1);
\coordinate (n3) at (2.6875,2);
\coordinate (n35) at (2.6875,3);
\coordinate (n37) at (0.75,4);
\coordinate (n48) at (0.75,5);
\coordinate (n50) at (0.0,6);
\coordinate (n53) at (0,7);
\coordinate (n49) at (1.5,6);
\coordinate (n52) at (1,7);
\coordinate (n51) at (2,7);
\coordinate (n36) at (4.625,4);
\coordinate (n39) at (3.5,5);
\coordinate (n45) at (3.5,6);
\coordinate (n47) at (3,7);
\coordinate (n46) at (4,7);
\coordinate (n38) at (5.75,5);
\coordinate (n41) at (5.0,6);
\coordinate (n44) at (5,7);
\coordinate (n40) at (6.5,6);
\coordinate (n43) at (6,7);
\coordinate (n42) at (7,7);
\coordinate (n2) at (12.65625,2);
\coordinate (n5) at (9.625,3);
\coordinate (n24) at (9.625,4);
\coordinate (n26) at (8.5,5);
\coordinate (n32) at (8.5,6);
\coordinate (n34) at (8,7);
\coordinate (n33) at (9,7);
\coordinate (n25) at (10.75,5);
\coordinate (n28) at (10.0,6);
\coordinate (n31) at (10,7);
\coordinate (n27) at (11.5,6);
\coordinate (n30) at (11,7);
\coordinate (n29) at (12,7);
\coordinate (n4) at (15.6875,3);
\coordinate (n7) at (13.75,4);
\coordinate (n18) at (13.75,5);
\coordinate (n20) at (13.0,6);
\coordinate (n23) at (13,7);
\coordinate (n19) at (14.5,6);
\coordinate (n22) at (14,7);
\coordinate (n21) at (15,7);
\coordinate (n6) at (17.625,4);
\coordinate (n9) at (16.5,5);
\coordinate (n15) at (16.5,6);
\coordinate (n17) at (16,7);
\coordinate (n16) at (17,7);
\coordinate (n8) at (18.75,5);
\coordinate (n11) at (18.0,6);
\coordinate (n14) at (18,7);
\coordinate (n10) at (19.5,6);
\coordinate (n13) at (19,7);
\coordinate (n12) at (20,7);
\end{scope}

\draw (n1) -- (n0);
\draw (n3) -- (n1);
\draw (n35) -- (n3);
\draw (n37) -- (n35);
\draw (n48) -- (n37);
\draw (n50) -- (n48);
\draw (n53) -- (n50);
\draw (n49) -- (n48);
\draw (n52) -- (n49);
\draw (n51) -- (n49);
\draw (n36) -- (n35);
\draw (n39) -- (n36);
\draw (n45) -- (n39);
\draw (n47) -- (n45);
\draw (n46) -- (n45);
\draw (n38) -- (n36);
\draw (n41) -- (n38);
\draw (n44) -- (n41);
\draw (n40) -- (n38);
\draw (n43) -- (n40);
\draw (n42) -- (n40);
\draw (n2) -- (n1);
\draw (n5) -- (n2);
\draw (n24) -- (n5);
\draw (n26) -- (n24);
\draw (n32) -- (n26);
\draw (n34) -- (n32);
\draw (n33) -- (n32);
\draw (n25) -- (n24);
\draw (n28) -- (n25);
\draw (n31) -- (n28);
\draw (n27) -- (n25);
\draw (n30) -- (n27);
\draw (n29) -- (n27);
\draw (n4) -- (n2);
\draw (n7) -- (n4);
\draw (n18) -- (n7);
\draw (n20) -- (n18);
\draw (n23) -- (n20);
\draw (n19) -- (n18);
\draw (n22) -- (n19);
\draw (n21) -- (n19);
\draw (n6) -- (n4);
\draw (n9) -- (n6);
\draw (n15) -- (n9);
\draw (n17) -- (n15);
\draw (n16) -- (n15);
\draw (n8) -- (n6);
\draw (n11) -- (n8);
\draw (n14) -- (n11);
\draw (n10) -- (n8);
\draw (n13) -- (n10);
\draw (n12) -- (n10);

\draw[-stealth, shorten >=1pt, line width = 1.2] (n1) -- ($(n1)!0.6!(n0)$);
\draw[-stealth, shorten >=1pt, line width = 1.2] (n3) -- ($(n3)!0.6!(n35)$);
\draw[-stealth, shorten >=1pt, line width = 1.2] (n35) -- ($(n35)!0.6!(n37)$);
\draw[-stealth, shorten >=1pt, line width = 1.2] (n37) -- ($(n37)!0.6!(n48)$);
\draw[-stealth, shorten >=1pt, line width = 1.2] (n48) -- ($(n48)!0.6!(n49)$);
\draw[-stealth, shorten >=1pt, line width = 1.2] (n50) -- ($(n50)!0.6!(n53)$);
\draw[-stealth, shorten >=1pt, line width = 1.2] (n49) -- ($(n49)!0.6!(n52)$);
\draw[-stealth, shorten >=1pt, line width = 1.2] (n36) -- ($(n36)!0.6!(n39)$);
\draw[-stealth, shorten >=1pt, line width = 1.2] (n39) -- ($(n39)!0.6!(n45)$);
\draw[-stealth, shorten >=1pt, line width = 1.2] (n45) -- ($(n45)!0.6!(n47)$);
\draw[-stealth, shorten >=1pt, line width = 1.2] (n38) -- ($(n38)!0.6!(n36)$);
\draw[-stealth, shorten >=1pt, line width = 1.2] (n41) -- ($(n41)!0.6!(n44)$);
\draw[-stealth, shorten >=1pt, line width = 1.2] (n40) -- ($(n40)!0.6!(n42)$);
\draw[-stealth, shorten >=1pt, line width = 1.2] (n2) -- ($(n2)!0.6!(n1)$);
\draw[-stealth, shorten >=1pt, line width = 1.2] (n5) -- ($(n5)!0.6!(n24)$);
\draw[-stealth, shorten >=1pt, line width = 1.2] (n24) -- ($(n24)!0.6!(n26)$);
\draw[-stealth, shorten >=1pt, line width = 1.2] (n26) -- ($(n26)!0.6!(n32)$);
\draw[-stealth, shorten >=1pt, line width = 1.2] (n32) -- ($(n32)!0.6!(n34)$);
\draw[-stealth, shorten >=1pt, line width = 1.2] (n25) -- ($(n25)!0.6!(n24)$);
\draw[-stealth, shorten >=1pt, line width = 1.2] (n28) -- ($(n28)!0.6!(n31)$);
\draw[-stealth, shorten >=1pt, line width = 1.2] (n27) -- ($(n27)!0.6!(n29)$);
\draw[-stealth, shorten >=1pt, line width = 1.2] (n4) -- ($(n4)!0.6!(n2)$);
\draw[-stealth, shorten >=1pt, line width = 1.2] (n7) -- ($(n7)!0.6!(n4)$);
\draw[-stealth, shorten >=1pt, line width = 1.2] (n18) -- ($(n18)!0.6!(n7)$);
\draw[-stealth, shorten >=1pt, line width = 1.2] (n20) -- ($(n20)!0.6!(n23)$);
\draw[-stealth, shorten >=1pt, line width = 1.2] (n19) -- ($(n19)!0.6!(n22)$);
\draw[-stealth, shorten >=1pt, line width = 1.2] (n6) -- ($(n6)!0.6!(n4)$);
\draw[-stealth, shorten >=1pt, line width = 1.2] (n9) -- ($(n9)!0.6!(n6)$);
\draw[-stealth, shorten >=1pt, line width = 1.2] (n15) -- ($(n15)!0.6!(n17)$);
\draw[-stealth, shorten >=1pt, line width = 1.2] (n8) -- ($(n8)!0.6!(n11)$);
\draw[-stealth, shorten >=1pt, line width = 1.2] (n11) -- ($(n11)!0.6!(n14)$);
\draw[-stealth, shorten >=1pt, line width = 1.2] (n10) -- ($(n10)!0.6!(n13)$);
\fill[black] (n0) circle (2pt);
\fill[red] (n1) circle (2pt);
\fill[blue] (n3) circle (2pt);
\fill[red] (n35) circle (2pt);
\fill[blue] (n37) circle (2pt);
\fill[red] (n48) circle (2pt);
\fill[blue] (n50) circle (2pt);
\fill[white, draw=black, line width = 0.7pt] (n53) circle (4pt);
\fill[red] (n49) circle (2pt);
\fill[white, draw=black, line width = 0.7pt] (n52) circle (4pt);
\fill[white, draw=black, line width = 0.7pt] (n51) circle (4pt);
\fill[red] (n36) circle (2pt);
\fill[blue] (n39) circle (2pt);
\fill[red] (n45) circle (2pt);
\fill[white, draw=black, line width = 0.7pt] (n47) circle (4pt);
\fill[white, draw=black, line width = 0.7pt] (n46) circle (4pt);
\fill[red] (n38) circle (2pt);
\fill[blue] (n41) circle (2pt);
\fill[white, draw=black, line width = 0.7pt] (n44) circle (4pt);
\fill[red] (n40) circle (2pt);
\fill[blue] (n43) circle (2pt);
\fill[white, draw=black, line width = 0.7pt] (n42) circle (4pt);
\fill[red] (n2) circle (2pt);
\fill[blue] (n5) circle (2pt);
\fill[red] (n24) circle (2pt);
\fill[blue] (n26) circle (2pt);
\fill[red] (n32) circle (2pt);
\fill[white, draw=black, line width = 0.7pt] (n34) circle (4pt);
\fill[white, draw=black, line width = 0.7pt] (n33) circle (4pt);
\fill[red] (n25) circle (2pt);
\fill[blue] (n28) circle (2pt);
\fill[white, draw=black, line width = 0.7pt] (n31) circle (4pt);
\fill[red] (n27) circle (2pt);
\fill[blue] (n30) circle (2pt);
\fill[white, draw=black, line width = 0.7pt] (n29) circle (4pt);
\fill[red] (n4) circle (2pt);
\fill[blue] (n7) circle (2pt);
\fill[red] (n18) circle (2pt);
\fill[white, draw=black, line width = 0.7pt] (n20) circle (4pt);
\fill[red] (n23) circle (2pt);
\fill[white, draw=black, line width = 0.7pt] (n19) circle (4pt);
\fill[blue] (n22) circle (2pt);
\fill[red] (n21) circle (2pt);
\fill[red] (n6) circle (2pt);
\fill[blue] (n9) circle (2pt);
\fill[white, draw=black, line width = 0.7pt] (n15) circle (4pt);
\fill[blue] (n17) circle (2pt);
\fill[red] (n16) circle (2pt);
\fill[red] (n8) circle (2pt);
\fill[white, draw=black, line width = 0.7pt] (n11) circle (4pt);
\fill[red] (n14) circle (2pt);
\fill[white, draw=black, line width = 0.7pt] (n10) circle (4pt);
\fill[blue] (n13) circle (2pt);
\fill[red] (n12) circle (2pt);
\end{tikzpicture}
\caption{\label{fig:frontier} $F_\rho(40)$ on the Fibonacci Tree, for $\rho(x) = x^{(\deg(x)-1)}$}
\end{figure}

\begin{lem}\label{lem:mn_bound}
There exists a constant $c < 1$ such that $M(n) < c n$, for all $n$ large enough.
\end{lem}
\begin{proof}
Let $x$ be an element of $F_{\rho}(n)$ with maximal distance $M = |x|$ to the root $r$. Denote by
$p = (r = x_0 , x_1, \ldots, x_M = x)$ the shortest path between $r$ and $x$. Since $F_{\rho}(1) = \{r\}$
and by the iterative construction of $F_{\rho}(n)$, there exist $1 = n_0 < n_1 < \cdots < n_M = n$, such that
$x_i \in F_{\rho}(n_i)$ for all $i\in 0,\ldots,M$. Since $\mathcal{T}$ is a directed cover that is not
isomorph to a single infinite path, it follows that for all $n$ big enough, there exist a constant
$\kappa > 0$, such that $\#\big\{v\in p: \deg(v) \geq 3\big\} \geq \kappa M$.

We want to find a lower bound for $n_{i+2} - n_i$, that is, for the number of steps needed to replace
$x_i$ by $x_{i+2}$ in the frontier. At time $n_i$, the vertex $x_i$ is added to the frontier.
The next time after $n_i$ that a particle visits $x_i$ halting condition \eqref{frontier_c} occurs,
thus the rotor at $x_i$ is incremented two times. As long as not all children of $x_i$
are part of the frontier, every particle can visit $x_i$ at most once, since it either stops immediately
at a child of $x_i$ on stopping condition \eqref{frontier_b} or is returned to the ancestor of $x_i$. Thus
at subsequent visits the rotor at $x_i$ is incremented exactly once. In order for $x_{i+2}$ to be added
to the frontier, the rotor at $x_i$ has to point at direction $x_{i+1}$ twice. Thus replacing $x_i$ with
$x_{i+2}$ in the frontier, needs at least $\deg (x_i)$ particles which visit $x_i$.
Hence, $n_{i+2} - n_i \geq \deg(x_i)$

We have
\begin{equation*}
\sum_{i=0}^{M-2} n_{i+2} - n_i = n_M + n_{M-1} - n_1 - n_0 < 2 n.
\end{equation*}
On the other hand, denote by
\begin{align*}
p_2 &= \#\big\{x_i : i \in \{0,\ldots, M-2\} \text{ s.t. } \deg(x_i) = 2\big\} \\
p_3 &= \#\big\{x_i : i \in \{0,\ldots, M-2\} \text{ s.t. } \deg(x_i) \geq 3\big\},
\end{align*}
then
\begin{align*}
\sum_{i=0}^{M-2} n_{i+2} - n_i & \geq \sum_{i=0}^{M-2} \deg(x_i)
    \geq 3 p_3 + 2 p_2 = 3 p_3 + 2(M-1 - p_3)\\
    &\geq p_3 + 2M - 2 \geq (\kappa + 2) M - 2\kappa-2.
\end{align*}
Thus $M \leq \frac{2}{\kappa+2} n + 2$, which proves the claim.
\end{proof}

\paragraph{The number of particles on the frontier.}
For the frontier process $F_{\rho}(n)$ defined above, when starting $n$ rotor particles at the root, we end up with exactly one particle at each vertex of $F_{\rho}(n)$ and the rest are in  $s_{\down}=r^{(0)}$ (the ancestor of the root). In order to obtain a lower bound for the cardinality of $F_{\rho}(n)$,
we first get an upper bound for the number of particles stopped at $s_{\down}$. This will be achieved using Theorem 1 from \cite{holroyd_propp}. Define
\begin{equation}\label{eq:fill_holes}
 \ell(n)=\{x\in\T_i: |x|=M(n) \text{ and the path from $r$ to $x$ contains no vertex of } F_{\rho}(n)\},
\end{equation}
where $M(n)$ is defined in \eqref{eq:max_height}.
By construction, the set $F_{\rho}(n)$ may have ``holes'': this means that $F_{\rho}(n)$ is not a cut in the tree. By introducing the set $\ell(n)$ in \eqref{eq:fill_holes}, we fill this holes by adding additional vertices on the maximal level $M(n)$. All these additional vertices were not touched by a rotor particle during the formation of $F_{\rho}(n)$. Fix $n$ and a rotor configuration $\rho$, and let 
\begin{equation}\label{eq:frontier_sink}
S=F_{\rho}(n)\cup \ell(n)
\end{equation}
be the sink determined by the frontier process $F_{\rho}(n)$.
Denote be $\T_i^{S}$ the finite tree which is obtained by truncating $\T_i$ at $S$, i.e.
$\T_i^S=\{x\in\T_i:\mathcal{C}(x)\cap S \neq \emptyset\} $.

Let $(X_t)$ be the simple random walk on $\T_i$.
Let $T_{s_{\down}}=\min\{t\geq 0: X_t\in s_{\down}\}$ and $T_S=\{t\geq 0: X_t\in S\}$ be the first hitting time of $s_{\down}$ and $S$ respectively.
Consider now the hitting probability
\begin{equation}\label{eq:hit_pb}
 h(x)=h_{s_{\down}}^{S}(x)=\P_x[T_{s_{\down}}<T_{S}],
\end{equation}
that is, the probability to hit $s_{\down}$ before $S$, when the random walk starts in $x$. 
We have $h\big(s_{\down})=1$ and $h(x)=0$, for all $x\in S$ and $h(x)=0$ for all $x\in\T_i\setminus\T_i^S$.  

Start now $n$ rotor particles at the root $r$, and stop them when they either reach $s_{\down}$ or $S$. By the Abelian property of rotor-router walks (see \cite[Lemma 24]{angel_holroyd_2012}) and by the construction of the frontier process $F_{\rho}(n)$ we will have exactly one rotor particle at each vertex of $F_{\rho}(n)$, no particles at $\ell(n)$, and the rest of the particles
are at $s_{\down}$. In order to estimate the proportion of rotor particles stopped at $s_{\down}$ we use Theorem 1 from \cite{holroyd_propp}, which we state here adapted to our case.
\begin{thm}[Theorem 1, \cite{holroyd_propp}]
Consider the sinks $s_{\down}$ and $S$ as above, and let $(X_t)$ be the simple random walk on $\T_i$.
Let $E$ be the set of edges of $\T_i$ and suppose that the quantity 
\begin{equation}\label{eq:const_k}
K=1+\sum_{(x,y)\in E}|h(x)-h(y)|
\end{equation}
is finite. If we start $n$ rotor particles at the root $r$,
then
\begin{equation}\label{eq:proportion_estimate}
 \Big|h(r)-\frac{n_{s_{\down}}}{n}\Big|\leq \frac{K}{n},
\end{equation}
where $n_{s_{\down}}$ represents the number of particles stopped at $s_{\down}$.
\end{thm}

\begin{lem}\label{lem:k_bound}
 The constant $K$ is equal to
 \begin{equation*}
 K= 1+\big(M(n)+1\big) \big(1-h(r)\big). 
 \end{equation*}
\end{lem}
\begin{proof}
The function $h$ is harmonic away from the sink:
$ h(x)=\frac{1}{\deg(x)}\sum _{y\sim x}h(y)$, if $x \notin s_{\down}\cup S$, and $h(x)=0$ for $x\in(\T_i\setminus \T_i^S)\cup S$. Therefore, there are only finitely many non zero summands in \eqref{eq:const_k}. For a vertex $x\in\T_i^{S}\setminus S$ and its ancestor $x^{(0)}$, we always have $h\big(x^{(0)}\big)\geq h(x)$, and
\begin{equation*}
h\big(x^{(0)}\big)-h(x)=\sum_{i=1}^{\deg(x)-1}\Big(h(x)-h\big(x^{(i)}\big)\Big).
\end{equation*}
Then 
\begin{equation*}
K=1+\Big(h\big(r^{(0)}\big)-h(r)\Big)+\sum_{k=0}^{M(n)-1}\sum_{x\in S^k}\sum_{i=1}^{\deg(x)-1}\Big(h(x)-h\big(x^{(i)}\big)\Big),
\end{equation*}
where $S^k=\{y\in\T_i:|y|=k\}$ represents the $k$-th level of the tree $\T_i$.
For a fixed $k$
\begin{align*}
\sum_{x\in S^k}\sum_{i=1}^{\deg(x)-1}\Big(h(x)-h\big(x^{(i)}\big)\Big) & 
= \sum_{x\in S^k}h\big(x^{(0)}\big)-h(x)=\sum_{x\in S^{k-1}}\sum_{i=1}^{\deg(x)-1}\Big(h(x)-h\big(x^{(i)}\big)\Big)\\
& = \sum_{x\in S^{k-j}}\sum_{i=1}^{\deg(x)-1}\Big(h(x)-h\big(x^{(i)}\big)\Big), \text{ for } j=2,\ldots ,k-1\\
& =h\big(r^{(0)}\big)-h(r)= 1 -h(r).
\end{align*}
Summing up over all levels the claim follows.
\end{proof}
\begin{cor}\label{cor:down_sink}
There is a constant $\kappa\in (0,1)$, such that $\#F_{\rho}(n)>\kappa n$, for all $n$ large enough.
\end{cor}
\begin{proof}
 From \eqref{eq:proportion_estimate}, we have $\frac{n_{s_{\down}}}{n}\leq \frac{K}{n} + h(r)$, where $n_{s_{\down}}$ is the number of particles stopping at $s_{\down}$ in the frontier process $F_{\rho}(n)$.
 Putting together Lemma \ref{lem:mn_bound} and \ref{lem:k_bound}, we obtain $K< 1+(cn+1) (1-h(r))$.
 Putting $\kappa'=h(r)(1-c)+c<1$, we get $n_{s_{\down}}<\kappa' n$. Since $\#F_{\rho}(n)=n-n_{s_{\down}}$,
 the claim follows.
\end{proof}

\begin{cor}\label{cor:escapes_decay}
Let $\rho$ be an initial random rotor configuration with distribution $\D=(\D_1,\ldots,\D_m)$ on 
the directed cover $\T_i$ with root of type $i$, of a finite strongly connected graph $\G$ with $m$ vertices. Suppose
$\r\left(M(\D)\right)>1$. Then there exists $\delta_i,c_i>0$, such that for all $n$
\begin{equation*}
 \P\big[E_n(\T_i,\rho)<\delta_i n \big]\leq e^{-c_in}, \quad \text{ for all }i\in\G.
\end{equation*} 
\end{cor}

\begin{proof}
Consider $n$ rotor walks particles and build the frontier process $F_{\rho}(n)$.
The MBP with probabilities $p^i$ as in equation \eqref{eq:mbp} and $\r\left(M(\D)\right)>1$ survives
with positive probability $\mathsf{p}_i$. Hence, for each $i\in\G$, with positive probability
there exists a live path starting at the root of $\T_i$. Existence of a live path
implies that the first particle escapes, hence
\begin{equation*}
\P[E_1(\T_i,\rho)=1]=\mathsf{p}_i>0,\quad \text{ for all } i\in\G.
\end{equation*}
Denote now by $X$ the set of vertices $x\in F_{\rho}(n)$, for which there is a live path starting at $x$.
Then $\# X=\sum_{x\in F_{\rho}(n)}Y_x$, where the random variables $Y_x\sim \text{Bernoulli}\big(\mathsf{p}_{\tau(x)}\big)$ are independent Bernoulli random variables.
Recall that $\tau(x)$ represents the type of the vertex $x$.
By the construction of $F_{\rho}(n)$, after starting $n$ rotor walks in the root $r$, we have exactly one rotor particle in each $x\in F_{\rho}(n)$. By Corollary \ref{cor:down_sink}, we have $\#F_{\rho}(n)>\kappa n$. Hence $\E[\#X]\geq \#F_{\rho}(n)\mathsf{p}>\kappa \mathsf{p} n $, where $\mathsf{p}=\min_{i}\mathsf{p}_i>0.$ Let us first prove that
\begin{equation}\label{eq:ldev}
E_n(\T_i,\rho)\geq \# X.
\end{equation}
From \cite[Lemmas 18,19]{holroyd_propp}, it suffices to prove
\eqref{eq:ldev} for the truncated tree $\T_i^H=\{x\in\T_i:|x|\leq H\}$, with $H>M(n)$, i.e.,
\begin{equation}\label{eq:esc_truncated}
 E_n(\T_i^H,S^H,\rho^H)\geq \# X.
\end{equation}
$E_n(\T_i^H,S^H,\rho^H)$ represents the number of particles that stop at $S^H=\{x\in\T_i:|x|=H\}$
when we start $n$ rotor-router walks at the root of
$\T_i$ and rotor configuration $\rho^H$ (the restriction of $\rho$ on $\T_i^H$).
In the tree $\T_i^S$, truncated at the frontier $S$, start $n$ particles at the root, and stop them when they either reach $S$ or return to $s_{\down}$. Moreover, the vertices
at distance greater than $M(n)$ were not reached, and the rotors there are unchanged.
Now for every vertex $x$ in $X$ restart one particle. Since there is a live path a $x$ the particle will
reach the level $H$ without leaving the cone of $x$, at which point the particle is stopped again.
Hence if we restart all particles which are located in $F_{\rho}(n)$ at least $\# X$ of them will reach
level $H$ before returning to the root. Because of the abelian property of rotor-router walks,
\eqref{eq:esc_truncated} follows, therefore also \eqref{eq:ldev}.

Using the Chernoff bound, there exists $\delta_i\in (0,1)$ such that
\begin{align*}
\P\big[E_n(\T_i,\rho) & <\delta_i n \big]  \leq \P\big[\#X< \delta_i n\big]
                     \leq \P\Big[\#X < \frac{\delta_i}{\kappa \mathsf{p}} \E[\#X]\Big]\\
 & \leq \exp\bigg\{-\frac{\big(1-\frac{\delta_i}{\kappa \mathsf{p}}\big)^2}{2}\E[\#X]\bigg\}
 \leq  \exp\bigg\{-\frac{\big(1-\frac{\delta_i}{\kappa \mathsf{p}}\big)^2}{2}\kappa \mathsf{p} n \bigg\}.
\end{align*}
We can then choose $c_i>0$ such that
\begin{equation*}
\P\big[E_n(\T_i,\rho) <\delta_i n \big]\leq  e^{-c_i n},
\end{equation*}
which proves the statement.
\end{proof}

We shall also need \cite[Lemma 25]{angel_holroyd_2011} which holds for general trees.
\begin{lem}
For a graph $G$ with $m$ vertices and $\T_i$ its directed cover with root $r$ of type
$i\in\G$, let $\T_{\chi_i(1)},\ldots,\T_{\chi_i(d_i)}$ be its principal branches rooted
at the children $r^{(k)}$ of the root, with $k=1,\ldots,d_i$.
Let $\rho$ be some rotor configuration on $\T_i$ and $\rho_k$ be its restriction on the tree
$\T_{\chi_i(k)}$. For each $i\in\G$, let
\begin{equation*}
 l_i=\liminf_{n\to\infty}\dfrac{E_n(\T_i,\rho)}{n} \quad \text{ and }\quad 
 l_i^k=\liminf_{n\to\infty}\dfrac{E_n(\T_{\chi_i(k)},\rho_k)}{n}, \quad k=1,\ldots,d_i.
\end{equation*}
Then 
\begin{equation}\label{eq:li_iteration}
 l_i\geq 1-\dfrac{1}{1+\sum_{k=1}^{d_i}l_i^k}=1-\dfrac{1}{1+\sum_{j\in\G}d_{ij}l_j}.
\end{equation}
\end{lem}
The probability $\mathcal{E}_i$ that a simple random walk $(X_t)$
on $\T_i$ never returns to the root satisfies an relation similar to \eqref{eq:li_iteration}. 
If  $r$ is the root of $\T_i$, then
$\mathcal{E}_i=\P_r[X_t\neq s_{\down},\forall t\geq 0]$.
Factorizing the random walk on $\T_i$ with respect to the first step, we get
\begin{equation*}
\big(1-\mathcal{E}_i\big)\Big(d_i+1-\sum_{j=1}^{m}d_{ij}\big(1-\mathcal{E}_j\big)\Big)=1,
\end{equation*}
which gives
\begin{equation}\label{eq:hitpb_iteration}
 \mathcal{E}_i=1-\dfrac{1}{1+\sum_{j\in\G} d_{ij}\mathcal{E}_j}
\end{equation}
We are now able to prove the transience part in  Theorem \ref{thm:rec/trans_rr_walks}.

\begin{proof}[Proof of Theorem \ref{thm:rec/trans_rr_walks}(b)]
For each $i\in\G$, let
\begin{equation*}
 l_i=\liminf_{n\to\infty}\dfrac{E_n(\T_i,\rho)}{n}.
\end{equation*}
Using Borel-Cantelli Lemma for the events in Corollary \ref{cor:escapes_decay},
it follows that for each $i$, there exists $\delta_i$ such that
\begin{equation*}
 \P\Big[\limsup_{n\to\infty}\dfrac{E_n(\T_i,\rho)}{n}<\delta_i\Big]=0.
\end{equation*}
Since
\begin{equation*}
\P\Big[\limsup_{n\to\infty}\dfrac{E_n(\T_i,\rho)}{n}<\delta_i\Big]=1-\P\Big[\liminf_{n\to\infty}\dfrac{E_n(\T_i,\rho)}{n}\geq\delta_i\Big],
\end{equation*}
we have $\P[l_i\geq\delta_i]=1$, with $\delta_i>0$.  Let $a_i$ be some positive constants 
such that $l_i\geq a_i$ for all $i$. Then
\begin{equation*}
l_i\geq 1-\dfrac{1}{1+\sum_{j\in\G} d_{ij}a_j} \text{ a.s.}
\end{equation*}
Applying this  repeatedly gives that for all $i$, $l_i$ 
is greater or equal to the fixed point of the iteration
\begin{equation*}
a_i\mapsto 1-\dfrac{1}{1+\sum_{j\in\G}d_{ij}a_j} 
\end{equation*}
from $\R^m\mapsto\R^m$. The return probabilities $\mathcal{E}_i$ are also solutions of the
same fixed point equation \eqref{eq:hitpb_iteration}. Hence $l_i\geq \mathcal{E}_i$.
On the other hand, by Theorem \ref{thm:schramm_bound} we have
$l_i\leq \mathcal{E}_i$, which implies 
\begin{equation*}
 \lim_{n\to\infty}\dfrac{E_n(\T_i,\rho)}{n}=\mathcal{E}_i, \text{ for all } i\in\G,
\end{equation*}
if $\r(M(\D))>1$, which proves the desired.
\end{proof}

\begin{rem}
For a supercritical positive regular multitype branching process,
in the event of nonextinction, the genealogical tree has branching number 
$\r(M)$. Like above, $M$ represents the matrix of the first moments of the MBP,
and $\r(M)$ its spectral radius.
\end{rem}
This means that in the transient case we have
$\r(M(\D))=\text{br}(\D)$, where $\text{br}(\D)$ is the branching number
of the genealogical tree with offspring distributions as given in \eqref{eq:mbp}.
For more information on the relation between the spectral radius of a branching 
process and the branching number, see {\sc Lyons} \cite{lyons_percolation}.

\subsubsection{Examples}

\begin{exam}
Let us first consider a {\em generalized Fibonacci tree} depending on the parameter $\alpha\in\N$. 
Consider the graph $\G$ with adjacency matrix
\begin{equation*}
 D=
 \begin{pmatrix}
  0 & \alpha \\
  1 & 1
 \end{pmatrix},
\end{equation*}
and $\T_i$ its directed cover with root of type $i=1,2$. Depending on the values 
of $\alpha$, the rotor-router walk on $\T_i$ can be either transient or recurrent.
If $\alpha=1$ we get the Fibonacci tree,  and for $\alpha=2$ we get the binary tree. 
On such trees we take a random initial configuration of rotors which is uniformly distributed on the
neighbours. Since these trees have $2$ types of vertices (first type $1$ with $\alpha$ children
of type $2$, and second type $2$ with $1$ child of type $1$ and one of type $2$), 
we have $\D=(\D_1,\D_2)$, with $\D_1=\text{Uniform}(0,\ldots,\alpha)$
and $\D_2=\text{Uniform}(0,1,2)$. Consider the following generation function $\chi_i$:
\begin{equation*}
 \chi_1(1)  = \ldots =\chi_1(\alpha)=2 \quad \text{ and } \quad
 \chi_2(1)  = 1 \text{ and }\chi_2(2)=2.
\end{equation*}
The transition probabilities $p^i$, $i=1,2$ defined in \eqref{eq:mbp}, which model a MBP consisting of only
good children are then given by:
\begin{equation*}
 p^1(0,0)  =\ldots = p^1(0,\alpha)=\frac{1}{\alpha}
 \quad \text{ and }\quad
 p^2(0,0) = p^2(1,0)=p^2(1,1)=\frac{1}{3}.
\end{equation*}
For the generating functions $f^i(z)$, with $z=(z_1,z_2)$ we then get:
\begin{align*}
f^1(z) &= \frac{1}{\alpha+1}\left(1 + z_2 + z_2^2 + \cdots + z_2^{\alpha}\right) \\
f^2(z) &= \frac{1}{3}(1+z_1+z_1 z_2).
\end{align*}
The behaviour of rotor-router walks on directed covers of graphs is controlled by the matrix $M(\D)$ of the
first moments of the MBP defined above. The entries of $M(\D)$ can be computed using \eqref{eq:MBP_first_moment},
and for this particular example we have
\begin{equation*}
 M(\D)=
\begin{pmatrix}
0 & \alpha/2 \\
2/3 & 1/3
\end{pmatrix}.
\end{equation*}
The spectral radius is $\r(M(\D)) = \frac{1}{6}\left(1+ \sqrt{12\alpha + 1}\right)$.
Therefore, the rotor-router walk is
\begin{itemize}
 \item recurrent for $\alpha\leq 2$
 \item transient for $\alpha>2$.
\end{itemize}
In particular, the rotor-router walk on the Fibonacci tree and on the
binary tree is recurrent. Note here the contrast with the simple random walk which is transient.
\end{exam}
The next example shows a case where different planar embeddings of the same tree, gives rise to changes
in the recurrence/transience of the rotor-walk with random initial rotor configuration.

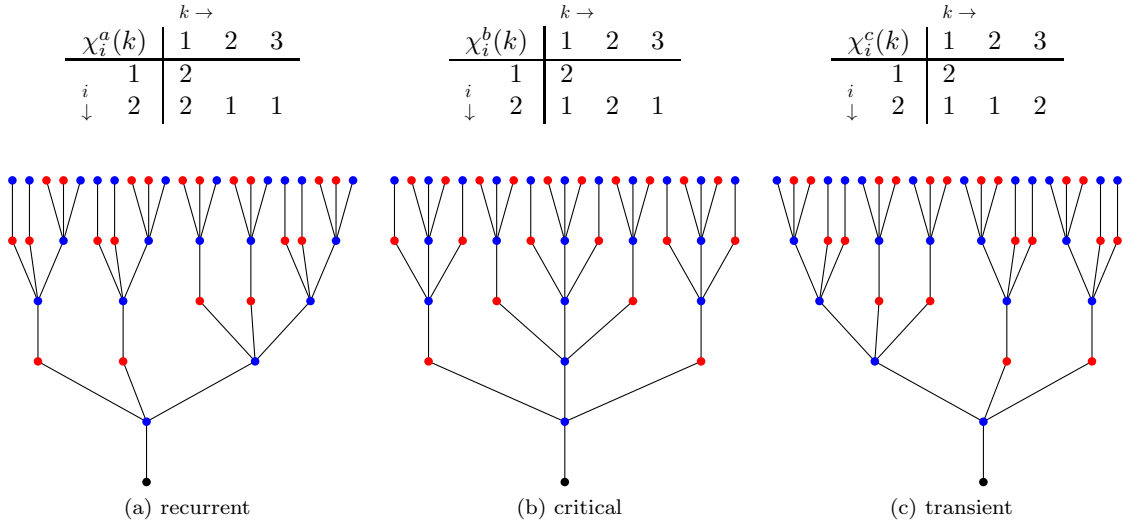
\begin{figure}
\begin{center}
\subfloat[\label{fig:permutation_example:recurrent}recurrent]{\begin{tikzpicture}[scale=0.8]
\begin{scope}[yshift=7cm]
\node (chi) at (2.8,0) {
\begin{tabular}{lc|ccc}
\multicolumn{2}{c}{} & \multicolumn{3}{l}{\scalebox{0.7}{$k\rightarrow$}} \\
\multicolumn{2}{c|}{$\chi^a_i(k)$} & 1 & 2 & 3 \\
\hline
\multirow{2}{*}{\scalebox{0.7}{\rotatebox{-90}{\rotatebox{90}{$i$} $\rightarrow$}}}
  & 1 & 2 & \\
  & 2 & 2 & 1 & 1
\end{tabular}
};
\end{scope}

\begin{scope}[xscale=0.28]
\coordinate (n0) at (7.875,0);
\coordinate (n1) at (7.875,1);
\coordinate (n4) at (1.5,2);
\coordinate (n33) at (1.5,3);
\coordinate (n36) at (0.0,4);
\coordinate (n41) at (0,5);
\coordinate (n35) at (1.0,4);
\coordinate (n40) at (1,5);
\coordinate (n34) at (3.0,4);
\coordinate (n39) at (2,5);
\coordinate (n38) at (3,5);
\coordinate (n37) at (4,5);
\coordinate (n3) at (6.5,2);
\coordinate (n24) at (6.5,3);
\coordinate (n27) at (5.0,4);
\coordinate (n32) at (5,5);
\coordinate (n26) at (6.0,4);
\coordinate (n31) at (6,5);
\coordinate (n25) at (8.0,4);
\coordinate (n30) at (7,5);
\coordinate (n29) at (8,5);
\coordinate (n28) at (9,5);
\coordinate (n2) at (14.25,2);
\coordinate (n7) at (11.0,3);
\coordinate (n20) at (11.0,4);
\coordinate (n23) at (10,5);
\coordinate (n22) at (11,5);
\coordinate (n21) at (12,5);
\coordinate (n6) at (14.0,3);
\coordinate (n16) at (14.0,4);
\coordinate (n19) at (13,5);
\coordinate (n18) at (14,5);
\coordinate (n17) at (15,5);
\coordinate (n5) at (17.5,3);
\coordinate (n10) at (16.0,4);
\coordinate (n15) at (16,5);
\coordinate (n9) at (17.0,4);
\coordinate (n14) at (17,5);
\coordinate (n8) at (19.0,4);
\coordinate (n13) at (18,5);
\coordinate (n12) at (19,5);
\coordinate (n11) at (20,5);
\end{scope}

\draw (n1) -- (n0);
\draw (n4) -- (n1);
\draw (n33) -- (n4);
\draw (n36) -- (n33);
\draw (n41) -- (n36);
\draw (n35) -- (n33);
\draw (n40) -- (n35);
\draw (n34) -- (n33);
\draw (n39) -- (n34);
\draw (n38) -- (n34);
\draw (n37) -- (n34);
\draw (n3) -- (n1);
\draw (n24) -- (n3);
\draw (n27) -- (n24);
\draw (n32) -- (n27);
\draw (n26) -- (n24);
\draw (n31) -- (n26);
\draw (n25) -- (n24);
\draw (n30) -- (n25);
\draw (n29) -- (n25);
\draw (n28) -- (n25);
\draw (n2) -- (n1);
\draw (n7) -- (n2);
\draw (n20) -- (n7);
\draw (n23) -- (n20);
\draw (n22) -- (n20);
\draw (n21) -- (n20);
\draw (n6) -- (n2);
\draw (n16) -- (n6);
\draw (n19) -- (n16);
\draw (n18) -- (n16);
\draw (n17) -- (n16);
\draw (n5) -- (n2);
\draw (n10) -- (n5);
\draw (n15) -- (n10);
\draw (n9) -- (n5);
\draw (n14) -- (n9);
\draw (n8) -- (n5);
\draw (n13) -- (n8);
\draw (n12) -- (n8);
\draw (n11) -- (n8);

\fill[black] (n0) circle (2pt);
\fill[blue] (n1) circle (2pt);
\fill[red] (n4) circle (2pt);
\fill[blue] (n33) circle (2pt);
\fill[red] (n36) circle (2pt);
\fill[blue] (n41) circle (2pt);
\fill[red] (n35) circle (2pt);
\fill[blue] (n40) circle (2pt);
\fill[blue] (n34) circle (2pt);
\fill[red] (n39) circle (2pt);
\fill[red] (n38) circle (2pt);
\fill[blue] (n37) circle (2pt);
\fill[red] (n3) circle (2pt);
\fill[blue] (n24) circle (2pt);
\fill[red] (n27) circle (2pt);
\fill[blue] (n32) circle (2pt);
\fill[red] (n26) circle (2pt);
\fill[blue] (n31) circle (2pt);
\fill[blue] (n25) circle (2pt);
\fill[red] (n30) circle (2pt);
\fill[red] (n29) circle (2pt);
\fill[blue] (n28) circle (2pt);
\fill[blue] (n2) circle (2pt);
\fill[red] (n7) circle (2pt);
\fill[blue] (n20) circle (2pt);
\fill[red] (n23) circle (2pt);
\fill[red] (n22) circle (2pt);
\fill[blue] (n21) circle (2pt);
\fill[red] (n6) circle (2pt);
\fill[blue] (n16) circle (2pt);
\fill[red] (n19) circle (2pt);
\fill[red] (n18) circle (2pt);
\fill[blue] (n17) circle (2pt);
\fill[blue] (n5) circle (2pt);
\fill[red] (n10) circle (2pt);
\fill[blue] (n15) circle (2pt);
\fill[red] (n9) circle (2pt);
\fill[blue] (n14) circle (2pt);
\fill[blue] (n8) circle (2pt);
\fill[red] (n13) circle (2pt);
\fill[red] (n12) circle (2pt);
\fill[blue] (n11) circle (2pt);
\end{tikzpicture}}
\hspace{0.2cm}
\subfloat[\label{fig:permutation_example:critical}critical]{\begin{tikzpicture}[scale=0.8]
\begin{scope}[yshift=7cm]
\node (chi) at (2.8,0) {
\begin{tabular}{lc|ccc}
\multicolumn{2}{c}{} & \multicolumn{3}{l}{\scalebox{0.7}{$k\rightarrow$}} \\
\multicolumn{2}{c|}{$\chi^b_i(k)$} & 1 & 2 & 3\\
\hline
\multirow{2}{*}{\scalebox{0.7}{\rotatebox{-90}{\rotatebox{90}{$i$} $\rightarrow$}}}
  & 1 & 2 & \\
  & 2 & 1 & 2 & 1
\end{tabular}
};
\end{scope}

\begin{scope}[xscale=0.28]
\coordinate (n0) at (10.0,0);
\coordinate (n1) at (10.0,1);
\coordinate (n4) at (2.0,2);
\coordinate (n33) at (2.0,3);
\coordinate (n36) at (0.0,4);
\coordinate (n41) at (0,5);
\coordinate (n35) at (2.0,4);
\coordinate (n40) at (1,5);
\coordinate (n39) at (2,5);
\coordinate (n38) at (3,5);
\coordinate (n34) at (4.0,4);
\coordinate (n37) at (4,5);
\coordinate (n3) at (10.0,2);
\coordinate (n16) at (6.0,3);
\coordinate (n29) at (6.0,4);
\coordinate (n32) at (5,5);
\coordinate (n31) at (6,5);
\coordinate (n30) at (7,5);
\coordinate (n15) at (10.0,3);
\coordinate (n23) at (8.0,4);
\coordinate (n28) at (8,5);
\coordinate (n22) at (10.0,4);
\coordinate (n27) at (9,5);
\coordinate (n26) at (10,5);
\coordinate (n25) at (11,5);
\coordinate (n21) at (12.0,4);
\coordinate (n24) at (12,5);
\coordinate (n14) at (14.0,3);
\coordinate (n17) at (14.0,4);
\coordinate (n20) at (13,5);
\coordinate (n19) at (14,5);
\coordinate (n18) at (15,5);
\coordinate (n2) at (18.0,2);
\coordinate (n5) at (18.0,3);
\coordinate (n8) at (16.0,4);
\coordinate (n13) at (16,5);
\coordinate (n7) at (18.0,4);
\coordinate (n12) at (17,5);
\coordinate (n11) at (18,5);
\coordinate (n10) at (19,5);
\coordinate (n6) at (20.0,4);
\coordinate (n9) at (20,5);
\end{scope}

\draw (n1) -- (n0);
\draw (n4) -- (n1);
\draw (n33) -- (n4);
\draw (n36) -- (n33);
\draw (n41) -- (n36);
\draw (n35) -- (n33);
\draw (n40) -- (n35);
\draw (n39) -- (n35);
\draw (n38) -- (n35);
\draw (n34) -- (n33);
\draw (n37) -- (n34);
\draw (n3) -- (n1);
\draw (n16) -- (n3);
\draw (n29) -- (n16);
\draw (n32) -- (n29);
\draw (n31) -- (n29);
\draw (n30) -- (n29);
\draw (n15) -- (n3);
\draw (n23) -- (n15);
\draw (n28) -- (n23);
\draw (n22) -- (n15);
\draw (n27) -- (n22);
\draw (n26) -- (n22);
\draw (n25) -- (n22);
\draw (n21) -- (n15);
\draw (n24) -- (n21);
\draw (n14) -- (n3);
\draw (n17) -- (n14);
\draw (n20) -- (n17);
\draw (n19) -- (n17);
\draw (n18) -- (n17);
\draw (n2) -- (n1);
\draw (n5) -- (n2);
\draw (n8) -- (n5);
\draw (n13) -- (n8);
\draw (n7) -- (n5);
\draw (n12) -- (n7);
\draw (n11) -- (n7);
\draw (n10) -- (n7);
\draw (n6) -- (n5);
\draw (n9) -- (n6);

\fill[black] (n0) circle (2pt);
\fill[blue] (n1) circle (2pt);
\fill[red] (n4) circle (2pt);
\fill[blue] (n33) circle (2pt);
\fill[red] (n36) circle (2pt);
\fill[blue] (n41) circle (2pt);
\fill[blue] (n35) circle (2pt);
\fill[red] (n40) circle (2pt);
\fill[blue] (n39) circle (2pt);
\fill[red] (n38) circle (2pt);
\fill[red] (n34) circle (2pt);
\fill[blue] (n37) circle (2pt);
\fill[blue] (n3) circle (2pt);
\fill[red] (n16) circle (2pt);
\fill[blue] (n29) circle (2pt);
\fill[red] (n32) circle (2pt);
\fill[blue] (n31) circle (2pt);
\fill[red] (n30) circle (2pt);
\fill[blue] (n15) circle (2pt);
\fill[red] (n23) circle (2pt);
\fill[blue] (n28) circle (2pt);
\fill[blue] (n22) circle (2pt);
\fill[red] (n27) circle (2pt);
\fill[blue] (n26) circle (2pt);
\fill[red] (n25) circle (2pt);
\fill[red] (n21) circle (2pt);
\fill[blue] (n24) circle (2pt);
\fill[red] (n14) circle (2pt);
\fill[blue] (n17) circle (2pt);
\fill[red] (n20) circle (2pt);
\fill[blue] (n19) circle (2pt);
\fill[red] (n18) circle (2pt);
\fill[red] (n2) circle (2pt);
\fill[blue] (n5) circle (2pt);
\fill[red] (n8) circle (2pt);
\fill[blue] (n13) circle (2pt);
\fill[blue] (n7) circle (2pt);
\fill[red] (n12) circle (2pt);
\fill[blue] (n11) circle (2pt);
\fill[red] (n10) circle (2pt);
\fill[red] (n6) circle (2pt);
\fill[blue] (n9) circle (2pt);
\end{tikzpicture}}
\hspace{0.2cm}
\subfloat[\label{fig:permutation_example:transient}transient]{\begin{tikzpicture}[scale=0.8]
\begin{scope}[yshift=7cm]
\node (chi) at (2.8,0) {
\begin{tabular}{lc|ccc}
\multicolumn{2}{c}{} & \multicolumn{3}{l}{\scalebox{0.7}{$k\rightarrow$}} \\
\multicolumn{2}{c|}{$\chi^c_i(k)$} & 1 & 2 & 3 \\
\hline
\multirow{2}{*}{\scalebox{0.7}{\rotatebox{-90}{\rotatebox{90}{$i$} $\rightarrow$}}}
  & 1 & 2 & \\
  & 2 & 1 & 1 & 2
\end{tabular}
};
\end{scope}

\begin{scope}[xscale=0.28]
\coordinate (n0) at (12.125,0);
\coordinate (n1) at (12.125,1);
\coordinate (n4) at (5.75,2);
\coordinate (n25) at (2.5,3);
\coordinate (n36) at (1.0,4);
\coordinate (n41) at (0,5);
\coordinate (n40) at (1,5);
\coordinate (n39) at (2,5);
\coordinate (n35) at (3.0,4);
\coordinate (n38) at (3,5);
\coordinate (n34) at (4.0,4);
\coordinate (n37) at (4,5);
\coordinate (n24) at (6.0,3);
\coordinate (n30) at (6.0,4);
\coordinate (n33) at (5,5);
\coordinate (n32) at (6,5);
\coordinate (n31) at (7,5);
\coordinate (n23) at (9.0,3);
\coordinate (n26) at (9.0,4);
\coordinate (n29) at (8,5);
\coordinate (n28) at (9,5);
\coordinate (n27) at (10,5);
\coordinate (n3) at (13.5,2);
\coordinate (n14) at (13.5,3);
\coordinate (n17) at (12.0,4);
\coordinate (n22) at (11,5);
\coordinate (n21) at (12,5);
\coordinate (n20) at (13,5);
\coordinate (n16) at (14.0,4);
\coordinate (n19) at (14,5);
\coordinate (n15) at (15.0,4);
\coordinate (n18) at (15,5);
\coordinate (n2) at (18.5,2);
\coordinate (n5) at (18.5,3);
\coordinate (n8) at (17.0,4);
\coordinate (n13) at (16,5);
\coordinate (n12) at (17,5);
\coordinate (n11) at (18,5);
\coordinate (n7) at (19.0,4);
\coordinate (n10) at (19,5);
\coordinate (n6) at (20.0,4);
\coordinate (n9) at (20,5);
\end{scope}

\draw (n1) -- (n0);
\draw (n4) -- (n1);
\draw (n25) -- (n4);
\draw (n36) -- (n25);
\draw (n41) -- (n36);
\draw (n40) -- (n36);
\draw (n39) -- (n36);
\draw (n35) -- (n25);
\draw (n38) -- (n35);
\draw (n34) -- (n25);
\draw (n37) -- (n34);
\draw (n24) -- (n4);
\draw (n30) -- (n24);
\draw (n33) -- (n30);
\draw (n32) -- (n30);
\draw (n31) -- (n30);
\draw (n23) -- (n4);
\draw (n26) -- (n23);
\draw (n29) -- (n26);
\draw (n28) -- (n26);
\draw (n27) -- (n26);
\draw (n3) -- (n1);
\draw (n14) -- (n3);
\draw (n17) -- (n14);
\draw (n22) -- (n17);
\draw (n21) -- (n17);
\draw (n20) -- (n17);
\draw (n16) -- (n14);
\draw (n19) -- (n16);
\draw (n15) -- (n14);
\draw (n18) -- (n15);
\draw (n2) -- (n1);
\draw (n5) -- (n2);
\draw (n8) -- (n5);
\draw (n13) -- (n8);
\draw (n12) -- (n8);
\draw (n11) -- (n8);
\draw (n7) -- (n5);
\draw (n10) -- (n7);
\draw (n6) -- (n5);
\draw (n9) -- (n6);

\fill[black] (n0) circle (2pt);
\fill[blue] (n1) circle (2pt);
\fill[blue] (n4) circle (2pt);
\fill[blue] (n25) circle (2pt);
\fill[blue] (n36) circle (2pt);
\fill[blue] (n41) circle (2pt);
\fill[red] (n40) circle (2pt);
\fill[red] (n39) circle (2pt);
\fill[red] (n35) circle (2pt);
\fill[blue] (n38) circle (2pt);
\fill[red] (n34) circle (2pt);
\fill[blue] (n37) circle (2pt);
\fill[red] (n24) circle (2pt);
\fill[blue] (n30) circle (2pt);
\fill[blue] (n33) circle (2pt);
\fill[red] (n32) circle (2pt);
\fill[red] (n31) circle (2pt);
\fill[red] (n23) circle (2pt);
\fill[blue] (n26) circle (2pt);
\fill[blue] (n29) circle (2pt);
\fill[red] (n28) circle (2pt);
\fill[red] (n27) circle (2pt);
\fill[red] (n3) circle (2pt);
\fill[blue] (n14) circle (2pt);
\fill[blue] (n17) circle (2pt);
\fill[blue] (n22) circle (2pt);
\fill[red] (n21) circle (2pt);
\fill[red] (n20) circle (2pt);
\fill[red] (n16) circle (2pt);
\fill[blue] (n19) circle (2pt);
\fill[red] (n15) circle (2pt);
\fill[blue] (n18) circle (2pt);
\fill[red] (n2) circle (2pt);
\fill[blue] (n5) circle (2pt);
\fill[blue] (n8) circle (2pt);
\fill[blue] (n13) circle (2pt);
\fill[red] (n12) circle (2pt);
\fill[red] (n11) circle (2pt);
\fill[red] (n7) circle (2pt);
\fill[blue] (n10) circle (2pt);
\fill[red] (n6) circle (2pt);
\fill[blue] (n9) circle (2pt);
\end{tikzpicture}}
\end{center}
\caption{\label{fig:permutation_example}Different planar embeddings of the same tree.}
\end{figure}

\begin{exam}
Consider a generating graph with $2$ vertex types and adjacency matrix
\begin{equation*}
D = \begin{pmatrix}
     0 & 1 \\
     2 & 1
    \end{pmatrix}.
\end{equation*}
There are three possible planar embeddings $\chi^a, \chi^b$ and $\chi^c$, which are shown in 
Figure \ref{fig:permutation_example} together with the directed covers they generate. On these trees
we perform transfinite rotor-router walks where the rotors are initially distributed according to
the uniform distribution on both types on vertices. 
The following table shows the generating functions,
the first moment matrix and spectral radius of the associated MBP in all cases.
\begin{center}
\scalebox{0.95}{
\begin{tabular}{c|lcl}
\multicolumn{1}{c}{} & generating function & $1^{\mathrm{st}}$ moment matrix & spectral radius \\
\midrule
\begin{minipage}[c]{2cm}
\begin{center}
$\chi^a$ \\[1ex]
Figure \ref{fig:permutation_example:recurrent}
\end{center}
\end{minipage}
& 
$\begin{aligned}
f^{1}(\boldsymbol{z}) &= \frac{1}{2} (z_2 + 1) \\[0.5ex]
f^{2}(\boldsymbol{z}) &= \frac{1}{4} (z_1^{2} z_2 + z_1^{2} + z_1 + 1)
\end{aligned}$ &
$\displaystyle M = \left(\begin{array}{rr} 0 & \frac{1}{2} \\[1ex] \frac{5}{4} & \frac{1}{4} \end{array}\right)$ &
$\r(M) = \frac{\sqrt{41} + 1}{8} < 1$ \\
\midrule
\begin{minipage}[c]{2cm}
\begin{center}
$\chi^b$ \\[1ex]
Figure \ref{fig:permutation_example:critical}
\end{center}
\end{minipage}
& 
$\begin{aligned}
f^{1}(\boldsymbol{z}) &= \frac{1}{2} (z_2 + 1) \\[0.5ex]
f^{2}(\boldsymbol{z}) &= \frac{1}{4} (z_1^2 z_2 + z_1 z_2 + z_1 + 1)
\end{aligned}$ &
$\displaystyle M = \left(\begin{array}{rr} 0 & \frac{1}{2} \\[1ex] 1 & \frac{1}{2} \end{array}\right)$ &
$\r(M) = 1$ \\
\midrule
\begin{minipage}[c]{2cm}
\begin{center}
$\chi^c$ \\[1ex]
Figure \ref{fig:permutation_example:transient}
\end{center}
\end{minipage}
& 
$\begin{aligned}
f^{1}(\boldsymbol{z}) &= \frac{1}{2} (z_2 + 1) \\[0.5ex]
f^{2}(\boldsymbol{z}) &= \frac{1}{4} (z_1^2 z_2 + z_1 z_2 + z_2 + 1)
\end{aligned}$ &
$\displaystyle M = \left(\begin{array}{rr} 0 & \frac{1}{2} \\[1ex] \frac{3}{4} & \frac{3}{4} \end{array}\right)$ &
$\r(M) = \frac{\sqrt{33} + 3}{8} > 1$ \\
\bottomrule
\end{tabular}
}
\end{center}
Hence depending only on the planar embedding the rotor-router walk is either recurrent (for $\chi^a$), recurrent in the critical case
(for $\chi^b$) or transient (for $\chi^c$). Another interpretation of this example
is: on the same tree (forgetting now about the planar embedding) different rotor sequence
gives rise to different behaviour of the rotor-router walk, for random initial configuration of rotors.
\end{exam}

\bibliography{rotor_walks}{}
\bibliographystyle{mypaperhep}

\end{document}